\documentclass[12pt]{amsart}
\usepackage{geometry}                
\usepackage{graphicx}
\usepackage{amssymb, enumerate,bm}
\usepackage[matrix,arrow,ps]{xy}
\usepackage{epstopdf}
\newcommand{\real}{{\rm Re}\,}
\newcommand{\imag}{{\rm Im}\,}
\newcommand{\C}{\mathbb C}
\newcommand{\R}{\mathbb R}

\newtheorem{theorem}{Theorem}[section]
\newtheorem{lemma}[theorem]{Lemma}

\theoremstyle{remark}
\newtheorem{remark}[theorem]{Remark}
\theoremstyle{example}

\theoremstyle{definition}

\numberwithin{equation}{section}

\begin{document}

\begin{abstract}
This paper is devoted to Riemann-Hilbert problems with constraints. We obtain results characterizing the existence of 
solutions as well as the dimension of the solution 
space in terms of certain indices. The results of this paper are particularly adapted to the study of stationary discs 
attached to CR manifolds.  
\end{abstract} 

\title[Riemann-Hilbert  problems with constraints]{Riemann-Hilbert  problems with constraints}

\author{Florian Bertrand and Giuseppe Della Sala}

\subjclass[2010]{32A30, 30E25, 35Q15}

\keywords{}
\maketitle

\section*{Introduction}

Very recently, the theory of stationary discs developed by Lempert \cite{le} (see also \cite{hu, ce1, tu}), has found important applications to the study of jet determination problems for CR maps between finitely smooth real submanifolds \cite{be-bl,  be-bl-me, be-de, be-de-la}.  The existence and geometric properties of stationary discs are deeply connected to 
non-linear Riemann-Hilbert problems. This connection was developed in the seminal works of Forstneri\v{c} \cite{fo} and Globevnik \cite{gl1,gl2} in their study of analytic discs attached to totally real submanifolds.   

In several applications \cite{be-bl-me, be-de-la}, one needs to study families of stationary discs passing through a prescribed point; in such cases, standard techniques developed in \cite{fo, gl1,gl2} cannot be applied directly and versions of Riemann-Hilbert problems with pointwise constraints are needed.  To the best of our knowledge, relevant results are not covered in the literature on Riemann-Hilbert problems. The present paper is devoted to this problem and our main results Theorem \ref{theorh} and Theorem \ref{theorh2} provide the tools required for the  construction of stationary discs with pointwise constraints as in \cite{be-bl-me, be-de-la}.

\vspace{0.5cm}

\noindent  {\it Acknowledgments.} Research of the two authors was supported by a fellowship at the Center for Advanced Mathematical Sciences (CAMS).  
The second author was partially supported by the Austrian Science Fund FWF grant P24878 N25.

\section{Preliminaries}    
Let $\Delta$ be the  unit disc in $\C$ and let $b\Delta$ be its boundary.

\subsection{Function spaces}
Let $k\geq 0$ be an integer and let $0< \alpha<1$.
We denote by $\mathcal C^{k,\alpha}=\mathcal C^{k,\alpha}(b\Delta,\R)$ the space of real-valued functions  defined on $b\Delta$ of class 
$C^{k,\alpha}$ endowed with its usual norm. We define $\mathcal C^{k,\alpha}_e$ (resp. $\mathcal C^{k,\alpha}_o$) to be the closed subspace of $\mathcal C^{k,\alpha}$ given by the even (resp. odd) functions, that is, the functions $v\in \mathcal C^{k,\alpha}$ such that $v(-\zeta) = v(\zeta)$ (resp. $v(-\zeta) = - v(\zeta)$ ) for all $\zeta\in b\Delta$.
We set $\mathcal C_\C^{k,\alpha} = \mathcal C^{k,\alpha} + i\mathcal C^{k,\alpha}$. Hence $v\in \mathcal C_\C^{k,\alpha}$ if and only if 
$\real v, \imag v \in \mathcal C^{k,\alpha}$. 
We denote by $\mathcal A^{k,\alpha}$ the subspace of $\mathcal C_{\C}^{k,\alpha}$ consisting of functions $f:\overline{\Delta}
\rightarrow \C$, holomorphic on $\Delta$ with trace on 
$b\Delta$ belonging to $\mathcal C_\C^{k,\alpha}$.

Let $m\geq 0$ be an integer. We define $\mathcal A^{k,\alpha}_{0^m}$
to be the subspace of 
$\mathcal C_{\C}^{k,\alpha}$  
consisting of the functions that can be written as $(1-\zeta)^m f$, with  $f\in \mathcal A^{k,\alpha}$. 
Finally, we denote by $\mathcal C_{0^m}^{k,\alpha}$ the subspace of $\mathcal C^{k,\alpha}$ consisting of elements that can be written as $(1-\zeta)^m v$ with $v\in \mathcal C_\C^{k,\alpha}$. 

For technical reasons we will also need  the following subspaces  of $\mathcal C_{\mathbb C}^{k,\alpha}$  
$$\mathcal R_m = \{v \in \mathcal C_{\mathbb C}^{k,\alpha} \ | \ v(\zeta) = (-1)^m \zeta^{-m} \overline {v(\zeta)} \ \ \forall \  \zeta \in b\Delta \}.$$ 
Their relation with $\mathcal C_{0^m}^{k,\alpha}$ is given by the following elementary lemma
\begin{lemma}\label{subspaces}
\
 \begin{enumerate}[(i)] 

\item  The map $\tau_m: \mathcal C_{0^m}^{k,\alpha} \to \mathcal C_{\mathbb C}^{k,\alpha}$ defined
 by $\tau_m ((1 - \zeta)^m v ) = v$ is an isomorphism between $\mathcal C_{0^m}^{k,\alpha}$ and $\mathcal R_m$;
\item if $m = 2m'$ is even, the map $v \mapsto \zeta^{m'}v$ induces an isomorphism between $\mathcal R_m$ and $\mathcal R_0 = \mathcal C^{k,\alpha}$;
\item if $m = 2m'+1$ is odd, the map $v \mapsto \zeta^{m'}v$ induces an isomorphism between $\mathcal R_m$ and $\mathcal R_1$.
\end{enumerate}
Furthermore, if $m$ is odd the map $v(\zeta) \mapsto  i \zeta^m v(\zeta^2)$ sends $\mathcal R_m$ isomorphically to $\mathcal C^{k,\alpha}_o$.
\end{lemma}
\begin{proof}
A function $v\in \mathcal C_{\mathbb C}^{k,\alpha}$ is in the image of $\tau_m$ exactly when $(1 - \zeta)^m v \in \mathcal C^{k,\alpha}$, that is, \ 
$$(1 - \zeta)^m v= (1 - \overline \zeta)^m \overline v = (-1)^m \zeta^{-m} (1 - \zeta)^m \overline v,$$ which gives the first point.

If $m = 2m'$ and $v\in \mathcal R_m$, $u = \zeta^{m'} v$, we have 
$$u = \zeta^{m'}v = \zeta^{m'} \zeta^{-2m'} \overline v = \zeta^{-m'}\overline v = \overline u,$$
 hence $u\in \mathcal C^{k,\alpha}$. Note that this series of equalities shows that  the map $v \mapsto \zeta^{m'}v$ is onto.

If $m = 2m'+1$ and $v\in \mathcal R_m$, $u = \zeta^{m'} v$, we have 
$$u = \zeta^{m'}v = -\zeta^{m'}\zeta^{-2m'+1} \overline v = - \overline \zeta \zeta^{-m'}\overline v = -\overline {\zeta u},$$ hence $u\in \mathcal R_1$.

Finally, letting $u(\zeta) = i \zeta^m v(\zeta^2)$ with $v\in \mathcal R_m$ and $m$ odd we have 
$$u(\zeta) = i \zeta^m v(\zeta^2) = -  i \zeta^m \zeta^{-2m} \overline {v(\zeta^2)} = -i \zeta^{-m} \overline {v(\zeta^2)} = 
\overline {u(\zeta)},$$
 and furthermore $u(-\zeta) = (-1)^m u(\zeta) = - u(\zeta)$, hence $u\in \mathcal C^{k,\alpha}_o$.
 
 We now show that this correspondence is an isomorphism by giving an explicit expression of the inverse. Write any $u\in \mathcal C^{k,\alpha}_o$ as
$$u(\zeta)=\sum_{j\in \mathbb Z} a_j \zeta^{2j+1} \mbox{ with } \overline{a_j}=a_{-j-1} \mbox{ for all } j\in \mathbb Z,$$ 
 and define
$$v(\zeta)= -\sum_{\ell\in \mathbb Z}ia_{\ell+(m-1)/2}\zeta^{\ell}.$$ 
Then
$$u(\zeta)/i\zeta^m= -\sum_{j\in \mathbb Z} ia_j \zeta^{2j-m+1}=-\sum_{\ell\in \mathbb Z}ia_{\ell+(m-1)/2}\zeta^{2\ell}= v(\zeta^2),$$
that is, $u(\zeta) = i \zeta^m v(\zeta^2)$. Moreover

$$-\zeta^{-m}\overline{v(\zeta)}= -\sum_{\ell\in \mathbb Z}i\overline a_{\ell+(m-1)/2}\zeta^{-\ell-m}= $$ $$= -\sum_{h\in \mathbb Z}i\overline a_{-h-(m+1)/2}\zeta^{h}=-\sum_{h\in \mathbb Z}i a_{h+(m-1)/2}\zeta^{h}=v(\zeta)$$ 
which shows that $v\in \mathcal R_m$.
 
\end{proof}

\subsection{Birkhoff factorization and indices}
We refer, for instance, to the monography of N.P. Vekua \cite{ve} for a complete exposition on the Birkhoff factorization and partial indices. We will recall the basic facts that we need. 
Let $N>0$ be an integer. We denote by $GL_N(\C)$ the general linear group on $\C^N$.  Let $G: b\Delta \to  GL_N(\C)$ be a smooth map. One considers a Birkhoff factorization  of $-\overline{G(\zeta)}^{-1}G(\zeta)$, that is, some smooth  maps $B^+:\bar{\Delta}\to GL_N(\C)$ and 
$B^-:(\C \cup \infty)\setminus\Delta\to GL_N(\C)$  such that for all $\zeta\in b\Delta$
$$-\overline{G(\zeta)}^{-1}G(\zeta)=
B^+(\zeta)
\left(\begin{array}{cccc}\zeta^{\kappa_1}& & & (0) \\ &\zeta^{\kappa_2} & & \\ & & \ddots & \\ (0)& & &\zeta^{\kappa_{N}}\end{array}\right)
B^-(\zeta)\,$$
where $B^+$ and $B^-$ are holomorphic on $\Delta$ and $\C \setminus \overline{\Delta}$ respectively. According to J. Globevnik (see Lemma 5.1 in \cite{gl1}), one can find $B^+$ and $B^-$ in such a way that 
$B^+=\Theta$ and  $B^-=\overline{\Theta^{-1}}$, where  $\Theta:\bar{\Delta}\to GL_N(\C)$ is a smooth map. 
The integers $\kappa_1, \dots, \kappa_N$  are called the {\it partial indices} of 
$-\overline{G}^{-1}G$ and  {\it the Maslov index} of $-\overline{G}^{-1}G$ is their sum 
$\kappa=\sum_{j=1}^N\kappa_j$. The partial indices are unique up to order (see for instance Section 3 in \cite{gl1}).  
We also recall that the Maslov index coincides with the winding number 
at the origin of the map $\det\left(-\overline{G}^{-1}G\right)$ and hence is even.

\section{Main results} 
\subsection{Linear Riemann-Hilbert problems with homogeneous pointwise constraints}

\begin{theorem}\label{theorh}
Let $k,m\geq 0$ be integers and let $0<\alpha<1$. 
Consider the following operator
\begin{equation*}
L: \left(\mathcal A^{k,\alpha}_{0^m}\right)^{N}
\to \left(\mathcal C_{0^m}^{k,\alpha}\right)^{N}
\end{equation*}
defined by
\begin{equation*}
L({\bm f})=2\real \left[\overline{G}{\bm f}\right],
\end{equation*}
where $G: b\Delta \to  GL_N(\C)$ is smooth.
Denote by $\kappa_1, \dots, \kappa_N$ and by $\kappa$ the  partial indices and  the Maslov index of 
$-\overline{G}^{-1}G$. Then

\begin{enumerate}[(i)] 
\item The map $L$ is onto if and only if $\kappa_j\geq m-1$ for all $j=1,\cdots,N$.
\item Assume that $L$ is onto. 
Then the kernel of $L$ has real dimension $\kappa+N-Nm$.   

\end{enumerate}
\end{theorem}
The proof will be given in Section 3.

\begin{remark}
In the context of Theorem \ref{theorh}, in case the partial indices $-\overline{G}^{-1}G$ are greater than or equal to $m-1$, it follows 
that $L$ is a Fredholm operator of Fredholm index $\kappa+N-Nm$, where again $\kappa$ is the Maslov index of  
$-\overline{G}^{-1}G$.
\end{remark}

\begin{remark}\label{remga}

The Riemann-Hilbert problem has important applications to the study of analytic discs attached to non degenerate real 
submanifolds of $\C^N$ \cite{le, fo, gl1, ce1, ce2}, and in particular to stationary discs; these are special analytics discs attached to a given 
hypersurface $M \subseteq \C^{n+1}$ which lift to the cotangent bundle as  discs - with a pole of order at most one - attached to the conormal bundle 
of $M$. In case $M$ is Levi-degenerate, its conormal bundle is no longer totally real and admits complex tangencies \cite{web}; the 
existence of smooth stationary discs attached to such a hypersurface and its  pertubations is therefore unclear. In \cite{be-de}, we have introduced 
generalized stationary discs by allowing the pole of their lifts to be of higher order. In order to construct generalized stationary discs attached to small 
pertubations of a Levi-degenerate hypersurface $M \subseteq \C^{n+1}$, one considers the corresponding Riemann-Hilbert problem whose linearization along an initial disc is of the form 
${\bm f} \mapsto 2\real [\overline{G}{\bm f}]$ where the matrix map $G(\zeta)$ is no longer invertible on $b\Delta$ (since the lift of the initial disc 
passes through complex points). A careful study of $G$ shows that one can factor out the singularities of $G$ and associate a related Riemann-Hilbert 
problem ${\bm f} \mapsto 2\real [\overline{\tilde{G}}{\bm f}]$, where $\tilde{G}(\zeta)$ is now invertible on $b\Delta$ and where ${\bm f}$ has new 
pointwise constraints. For more details and a clear application of Theorem \ref{theorh}, the reader is invited to see Theorem 4.2 in \cite{be-de-la} and 
its proof.

%
%
%
%
%
%
%

\end{remark}

\subsection{Linear Riemann-Hilbert problems with pointwise constraints}
Let $G: b\Delta \to  GL_N(\C)$ be a smooth map of the form
$$
G(\zeta)=\left(\begin{array}{cccc}G_1(\zeta)& & & (*) \\ &G_2(\zeta) & & \\ & & \ddots & \\ (0)& & &G_r(\zeta)\end{array}\right),
$$
where $G_j(\zeta) \in GL_{N_j}(\C)$ for all $j=1,\cdots,r$, for all $\zeta \in b\Delta$ and where $N_j$'s are positive integers such that their sum is $N$. Let $k,m_1,\ldots,m_r\geq 0$ be integers and let $0<\alpha<1$. Consider the following operator 
$$L: \prod_{l=1}^{r} \left(\mathcal{A}^{k,\alpha}_{0^{m_j}}\right)^{N_j}
\to \prod_{l=1}^{r} \left(\mathcal C_{0^{m_j}}^{k,\alpha}\right)^{N_j}$$
defined by
\begin{equation*}
L({\bm f})=2\real \left[\overline{G}{\bm f}\right].
\end{equation*}
Note that we are implicitly assuming that $G$ is such that $L\left(\prod_{l=1}^{r} \left(\mathcal{A}^{k,\alpha}_{0^{m_j}}\right)^{N_j}\right)\subset \prod_{l=1}^{r} \left(\mathcal C_{0^{m_j}}^{k,\alpha}\right)^{N_j}$. 
Denote by $\tilde{G}(\zeta)$ the following matrix
$$
\tilde{G}(\zeta)=\left(\begin{array}{cccc}G_1(\zeta)& & & (0) \\ &G_2(\zeta) & & \\ & & \ddots & \\ (0)& & &G_r(\zeta)\end{array}\right)
$$
and by $\tilde{L}$ the corresponding operator. For  $j=1,\cdots,r$ we denote by $\kappa_{l}^{j}$, $l=1,\cdots, N_j$, the partial indices of  $-\overline{G_j}^{-1}G_j$ and by $\kappa$ the 
Maslov index of $-\overline{G}^{-1}G$ and of $-\overline{\tilde{G}}^{-1}\tilde{G}$. Note that the fact that $\tilde{L}$ is onto implies that $L$ is onto and also  that  the kernels of $L$ and  $\tilde{L}$ are of the same dimension. 
A direct application of  Theorem \ref{theorh} gives:
\begin{theorem}\label{theorh2}
Under the above assumptions:
\begin{enumerate}[(i)]
\item If $\kappa_{l}^j \geq m_j-1$ for all \ $l=1,\cdots,N_j$  
and all $j=1,\cdots,r$ then the map $L$ is onto.  
\item Assume that $L$ is onto. Then the kernel of $L$ has real dimension $\kappa+N-\sum_{j=1}^{r}N_jm_j$. 
\end{enumerate}
\end{theorem}

\begin{remark}
In \cite{be-bl-me}, the first author, L. Blanc-Centi and F. Meylan construct classical stationary discs attached to pertubations of a given non-degenerate generic quadric and satisfying a pointwise constraint (the latter is essential for applications to jet determination problems). This construction is done using Theorem \ref{theorh2} (see Theorem 3.1 in \cite{be-bl-me} and its proof for more details).  
\end{remark}

\section{Proof of Theorem \ref{theorh}}

We start with a few more observations.
Following \cite{gl2}, one can find a smooth map $V:b\Delta \to GL_N(\R)$ such that
$V\overline {G} = M(-i\Theta)^{-1}$, where $M$ has a special block-diagonal form (see (\ref{eqM})), where 
$\Theta: \overline{\Delta} \to GL_N(\C)$ is smooth and holomorphic on $\Delta$, and in particular 
\begin{equation*}
\begin{array}{lll}
-\overline{G(\zeta)}^{-1}G(\zeta)& = &  \Theta(\zeta)M(\zeta)^{-1}\overline{M(\zeta)} \overline{\Theta(\zeta)^{-1}}\\
\\
& = &\Theta(\zeta)\left(\begin{array}{cccc}\zeta^{\kappa_1}& & & (0) \\ &\zeta^{\kappa_2} & & \\ & & \ddots & \\ (0)& & &\zeta^{\kappa_{N}}\end{array}\right)\overline{\Theta(\zeta)}^{-1}.\\
\\
\end{array}
\end{equation*}
Define the operator 
$$  
\tilde{L}: \left(\mathcal A^{k,\alpha}_{0^m}\right)^{N}
\to \left(\mathcal C_{0^{m}}^{k,\alpha}\right)^{N}$$
by setting
\begin{equation*}
\tilde{L}({\bm f})=2\real \left[M{\bm f}\right].
\end{equation*}
Since $\Theta: \overline{\Delta} \to GL_N(\C)$ is smooth and holomorphic on $\Delta$, the map 
$(i\Theta)^{-1}$ is an   isomorphism of  
$ \left(\mathcal A^{k,\alpha}_{0^m}\right)^{N}
$ onto itself. Moreover since $V$ is valued in $GL_N(\R)$, the map $\left(\mathcal C_{0^{m}}^{k,\alpha}\right)^{N} \to \left(\mathcal C_{0^{m}}^{k,\alpha}\right)^{N}$ defined by ${\bm \varphi} \mapsto V{\bm \varphi}$ is also an isomorphism. Therefore the kernels of $L$ and  $\tilde{L}$ are of the same dimension and 
$L$ is onto if and only if $\tilde{L}$ is onto; the operators $L$ and  $\tilde{L}$ are both Fredholm and of the same index. We will prove Theorem \ref{theorh} for $\tilde{L}$. 
\vspace{0.5cm}

We first prove $(i)$. Since $\kappa$ is even, the number of odd partial indices is even. Without loss of generality, suppose that $\kappa_j$ is odd for $j=1,\cdots,2s$ and that $\kappa_j$ is even for 
$j=2s+1,\cdots,N$. According to \cite{gl2}, the matrix $M$ can be written as  
\begin{equation}\label{eqM}
M(\zeta)=\left(\begin{array}{ccccccc}P_1(\zeta)& & &  & &  & (0) \\  &   \ddots & & & & &  \\ & & &P_{s}(\zeta)&  & &  \\
& & & &  \zeta^{-\frac{\kappa_{2s+1}}{2}}  &  & \\ &  & & &    &  \ddots & \\ (0) &  & & & & &  
\zeta^{-\frac{\kappa_N}{2}}\\
\end{array}\right),
\end{equation}
where 
$$P_j(\zeta)= \left(\begin{array}{cc} 1+\zeta & -i(1-\zeta) \\  i(1-\zeta)& 1+\zeta \end{array}\right)
\left(\begin{array}{cc} \zeta^{-\frac{\kappa_{2j-1}+1}{2}} & 0 \\ 0 & \zeta^{-\frac{\kappa_{2j}+1}{2}} \end{array}\right)$$ 
for $j=1,\cdots,s$. Note that in case there are no odd partial indices, the matrix $M(\zeta)$ is diagonal with entries $\zeta^{-\frac{\kappa_j}{2}}$. Part $(i)$ is a consequence of the following two lemmas

\begin{lemma}\label{lemsur1}
Let $r$ be an integer. Consider the operator $L: \mathcal A^{k,\alpha}_{0^m} \to 
\mathcal{C}_{0^m}^{k,\alpha}$ defined by 
$$L(f)=2\real [\zeta^{-r}f]|_{b\Delta}.$$
Then $L$ is onto if and only if  $2r \geq m-1$. 
\end{lemma}
\begin{proof}
Let $\varphi \in \mathcal{C}_{0^m}^{k,\alpha}$. Write $\varphi=(1-\zeta)^{m}v$ where  $v\in \mathcal R_m$ (see Lemma \ref{subspaces}). We need to study the following equation:
\begin{equation*}
\zeta^{-r} f +  \zeta^{r} \overline f = \varphi
\end{equation*}
for $f \in \mathcal A^{k,\alpha}_{0^m}$. Writing $f=(1-\zeta)^mg$ with $g \in \mathcal{A}^{k,\alpha}$ reduces to
\begin{equation}\label{eqsur}
\zeta^{-r} g +(-1)^m \zeta^{r-m}\overline g = v.
\end{equation}   
We distinguish two cases:

\noindent \underline{First case:} $m=2m'$ is even. 
In such case, Equation (\ref{eqsur}) is equivalent to
\begin{equation}\label{eqsureven}
\zeta^{-(r-m')} g + \zeta^{r-m'}\overline g = \zeta^{m'}v.
\end{equation}   
Notice that  $u\to \zeta^{m'}u$ maps $\mathcal R_m$ isomorphically to $\mathcal C^{k,\alpha}$, see Lemma \ref{subspaces}. Equation (\ref{eqsureven}) is classical 
and was treated by J. Globevnik in  \cite{gl2} for instance (see also \cite{ga,ve,we}). Equation (\ref{eqsureven}) is solvable for any function $v \in \mathcal R_m$ if and only if $r-m'\geq0$.

\noindent \underline{Second case:} $m=2m'+1$ is odd. In this case, Equation (\ref{eqsur}) is equivalent to
\begin{equation}\label{eqsurodd}
\zeta^{-(r-m')} g - \zeta^{r-m'} \overline {\zeta g} = \zeta^{m'}v.
\end{equation}   
Set $u=\zeta^{m'}v$. By Lemma \ref{subspaces} follows that $u \in \mathcal R_1$. We write $u = u' + u''$, where $u'= \mathcal P(u)\in \mathcal A^{k,\alpha}$, $\mathcal P$ being the Szeg\"o projection. Since $u=-\overline{\zeta u}$, we have $u'' = -\overline{\zeta u'}$. If $r-m' \geq 0$, then 
$g= \zeta^{r-m'}u' \in \mathcal A^{k,\alpha}$ and satisfies (\ref{eqsurodd}). 
If $r-m' < 0$, then 
$$\int \zeta^{-(r-m')} g \ d\theta = \int \zeta^{r-m'-1}\overline g \ d\theta = 0$$ and so, for 
instance, $1 - \overline \zeta \in \mathcal R_1$ is not in the image.
\end{proof}
 
\begin{lemma}\label{lemsur2}
Let $r_1,r_2$ be integers. Set 
$$P(\zeta)= \left(\begin{array}{cc} 1+\zeta & -i(1-\zeta) \\  i(1-\zeta)& 1+\zeta \end{array}\right)
\left(\begin{array}{cc} \zeta^{-r_1} & 0 \\ 0 & \zeta^{-r_2} \end{array}\right).$$ 
Consider the operator $T: \left(\mathcal A^{k,\alpha}_{0^m}\right)^2 \to 
\left(\mathcal{C}_{0^m}^{k,\alpha}\right)^2$ defined by 
$$T({\bm f})=2\real [P{\bm f}]|_{b\Delta}.$$
Then $T$ is onto if and only if  $2r_1\geq m$ and $2r_2\geq m$. 
\end{lemma}
\begin{proof}
Let ${\bm \varphi} \in \left(\mathcal{C}_{0^m}^{k,\alpha}\right)^2$. Write $\bm \varphi=(1-\zeta)^{m}\bm v$ where  
$\bm v\in \left(\mathcal{C}_{\C}^{k,\alpha}\right)^2$ and ${\bm f}=(1-\zeta)^m{\bm g}$ with ${\bm g} \in \left(\mathcal{A}^{k,\alpha}\right)^2$. We need to study the following equation:
\begin{equation*}
P {\bm g} +(-1)^m \zeta^{-m}\overline {P{\bm g}} = \bm v.
\end{equation*}   

\noindent \underline{First case:} $m=2m'$ is even. In that case, we have 
\begin{equation}\label{eqsur'2}
\zeta^{m'}P {\bm g} +\zeta^{-m'}\overline {P{\bm g}} = \zeta^{m'} \bm v,
\end{equation}  
which was treated by J. Globevnik \cite{gl2}. In particular, (\ref{eqsur'2}) admits a solution if and only if $r_1-m'\geq 0$ and $r_2-m'\geq 0$. 

\noindent \underline{Second case:} $m=2m'+1$ is odd. We have
\begin{equation*}
\zeta^{m'}P {\bm g} - \zeta^{-m'-1}\overline {P{\bm g}} = \zeta^{m'}\bm v.
\end{equation*}  
Following J. Globevnik, we make the substitution $\zeta = \xi^2$ and get 
\begin{equation*}
\xi^{m}P(\xi^2) {\bm g}(\xi^2) - \xi^{-m}\overline {P(\xi^2){\bm g}(\xi^2)} = \xi^{m}\bm v(\xi^2),
\end{equation*} 
After multiplying by $i$
\begin{equation*}
2 \real \left[ \xi^{m}P(\xi^2) i{\bm g}(\xi^2) \right] = i\xi^{m}\bm v(\xi^2),
\end{equation*} 
which becomes
\begin{equation}\label{eqsur'4}
4 \real \left[
 \left(\begin{matrix}i \xi^{-(2r_1-m-1)} g_1(\xi^2)    \\
i \xi^{-(2r_2-m-1)} g_2(\xi^2)  \\
\end{matrix}\right) \right] = i\left(\begin{array}{cc} \real \xi  & \imag \xi \\ -\imag \xi & \real \xi \end{array}\right)\xi^{m}\bm v(\xi^2).
\end{equation} 
Notice that, according to Lemma \ref{subspaces}, whenever $u \in \mathcal{R}_m$ then $2i\xi^mu(\xi^2)\in \mathcal C^{k,\alpha}_o$ and that moreover the map $\bm u \mapsto \left(\begin{matrix}
\real \xi &   \imag \xi      \\

-\imag\xi  &  \real\xi   \\
\end{matrix}\right)\bm u $ is an isomorphism between $(\mathcal C^{k,\alpha}_o)^2$ and $(\mathcal C^{k,\alpha}_e)^2$. Thus, (\ref{eqsur'4}) reduces to a pair of one-dimensional problems
\[\xi^{-(2r_j - m - 1)} g_j(\xi^2) + \xi^{2r_j-m-1}\overline g_j(\xi^2) = u_j(\xi)\]
with $u_j\in \mathcal C^{k,\alpha}_e$, $j=1,2$. Writing $u_j(\xi) = u_j'(\xi^2)$ with 
$u'_j\in \mathcal C^{k,\alpha}$, this equation is in turn equivalent to
\[\zeta^{-(2r_j - m - 1)/2} g_j(\zeta) + \zeta^{(2r_j-m-1)/2}\overline g_j(\zeta) = u'_j(\zeta).\] 
This problem is of the form considered in Lemma \ref{lemsur1}, and the surjectivity is equivalent to $2r_1-m -1\geq 0$, and $2r_2-m -1\geq 0$. Since 
$m$ is odd, this concludes the proof.
\end{proof}

\vspace{0.5cm}
We now prove $(ii)$.   
Assume that 
$$2\real \left[M{\bm f}\right]=0$$
on $b \Delta$. 
The disc ${\bm f} \in \left(\mathcal A^{k,\alpha}_{0^m}\right)^{N}$ satisfies 
$${\bm f}=-M^{-1}\overline{M}\overline{{\bm f}}=-\left(\begin{array}{cccc}\zeta^{\kappa_1}& & & (0) \\ &
\zeta^{\kappa_2} & & \\ & & \ddots & \\ (0)& & &\zeta^{\kappa_{N}}\end{array}\right)\overline{{\bm f}}.$$
The determination of the kernel thus reduces to the one dimensional problem 
\begin{equation}\label{eqrh0}
f+\zeta^l\overline{f}=0.
\end{equation}
for $f=(1-\zeta)^m g \in \mathcal{A}^{k,\alpha}_{0^m} $ and $l\geq m-1$. 
This equation can be written as 
$$
g+(-1)^m\zeta^{l-m}\overline{g}=0.
$$
It is immediate and classical (see for instance \cite{ga,gl2,ve,we}) that solutions are of the form 
$g(\zeta)=\sum_{k=0}^{l-m} a_k \zeta^k$ with $a_k=(-1)^{m+1}\overline{a_{l-m-k}}$ for $k=0,\cdots,l-m$. Therefore the space of solutions of (\ref{eqrh0}) has real dimension $l+1-m.$  This ends the proof of Theorem \ref{theorh}.

\qed
       
\begin{remark}
In relation with Theorem \ref{theorh} and its proof, note the work of M. \v{C}erne  \cite{ce2} in the framework of non-trivial bundles over the boundary of a given disc, that is when the corresponding Maslov index is odd. 

\end{remark}

\vskip 1cm
{\small
\noindent Florian Bertrand\\
Department of Mathematics, Fellow at the  Center for Advanced Mathematical Sciences\\
American University of Beirut, Beirut, Lebanon\\{\sl E-mail address}: fb31@aub.edu.lb\\

\noindent Giuseppe Della Sala \\
Department of Mathematics, Fellow at the  Center for Advanced Mathematical Sciences\\
American University of Beirut, Beirut, Lebanon\\{\sl E-mail address}: 	gd16@aub.edu.lb\\
}

\end{document}